\documentclass[a4paper, 11pt,reqno]{amsart}
\usepackage[top = 1in, bottom = 1in, left=1.3in, right=1.3in]{geometry}

\usepackage[utf8]{inputenc}
\usepackage{amsmath}
\usepackage{amsfonts}
\usepackage{amssymb}
\usepackage{caption}
\usepackage{dsfont}

\usepackage{amsthm,amsmath,amssymb}
\usepackage{verbatim}
\usepackage{ centernot}
\usepackage[shortlabels]{enumitem} 
\setlist[enumerate]{label={\upshape{(\roman*)}}}

\newtheorem{theorem}{Theorem}
\newtheorem{lemma}[theorem]{Lemma}

\newtheorem{claim}[theorem]{Claim}

\theoremstyle{definition}

\newtheorem{conjecture}[theorem]{Conjecture}

\newtheorem{question}[theorem]{Question}

\theoremstyle{remark}

\usepackage[
backend=biber,
natbib=true,
style=alphabetic,
isbn=false,
maxbibnames=9
]{biblatex}
\usepackage{csquotes}
\renewbibmacro{in:}{}
\DeclareFieldFormat{pages}{#1}

\usepackage{xcolor}
\usepackage{hyperref}
\hypersetup{colorlinks, linkcolor={red!50!black}, citecolor={green!50!black}, urlcolor={blue!50!black}}
\usepackage[capitalise]{cleveref}

\usepackage{xurl}
\hypersetup{breaklinks=true}

\DeclareFieldFormat{title}{\mkbibquote{#1}}
\DeclareFieldFormat[misc]{date}{Preprint (#1)}

\renewbibmacro*{doi+eprint+url}{
	\iftoggle{bbx:url}
	{\iffieldundef{doi}{\usebibmacro{url+urldate}}{}}{}
	\newunit\newblock
	\iftoggle{bbx:eprint}
	{\usebibmacro{eprint}}
	{}
	\newunit\newblock
	\iftoggle{bbx:doi}
	{\printfield{doi}}
	{}
}

\allowdisplaybreaks

\def\leq{\leqslant}
\def\geq{\geqslant}
\def\le{\leqslant}
\def\ge{\geqslant}

\def\mathbb{\mathds}

\title[Ramsey numbers with prescribed rate of growth]{Ramsey numbers with prescribed rate of growth}

\author[M.~Pavez-Sign\'e]{Mat\'ias Pavez-Sign\'e}
\address[M.~Pavez-Sign\'e]{Mathematics Institute, Zeeman Building, University of Warwick, Coventry, CV4 7AL, United Kingdom.}
\email{matias.pavez-signe@warwick.ac.uk}

\author[S.~Piga]{Sim\'on Piga}
\address[S.~Piga]{School of Mathematics, University of Birmingham, Birmingham, B15 2TT, United
Kingdom.}
\email{s.piga@bham.ac.uk}

\author[N.~Sanhueza-Matamala]{Nicol\'as Sanhueza-Matamala}
\address[N.~Sanhueza-Matamala]{Departamento de Ingeniería Matemática, Facultad de Ciencias Físicas y Matemáticas, Universidad de Concepción, Chile.}
\email{nicolas@sanhueza.net}
\thanks{MPS was supported by the European Research Council grant 947978 under the European Union’s Horizon 2020 research and innovation programme. SP was supported by EPSRC grant~EP/V002279/1. NSM was partly supported by ANID-Chile through the FONDECYT Iniciación Nº11220269 grant.
There are no additional data beyond that contained within the main manuscript.
}
\begin{document}

\begin{abstract}
    Let $R(G)$ be the $2$-colour Ramsey number of a graph $G$.
    In this note, we prove that for any non-decreasing function $n \leq f(n) \leq R(K_n)$, there exists a sequence of connected graphs $(G_n)_{n\in\mathbb N}$, with $|V(G_n)| = n$ for all $n \geq 1$,  such that $R(G_n) = \Theta(f(n))$.
    In contrast, we also show that an analogous statement does not hold for hypergraphs of uniformity at least~$5$.
    
    We also use our techniques to answer in the affirmative a question posed by DeBiasio about the existence of sequences of graphs, but whose $2$-colour Ramsey number is linear whereas their $3$-colour Ramsey number has superlinear growth.
\end{abstract}

\maketitle

\section{Introduction}
For a graph $G$ and $r \geq 2$, the \emph{$r$-colour Ramsey number} $R_r(G)$ of $G$ is the smallest number $n$ such that every $r$-colouring of the edges of the complete graph $K_n$ contains a monochromatic copy of $G$, that is, a copy of $G$ with all its edges in the same colour.
For $r = 2$, we will simply write $R_2(G) = R(G)$ and refer to this as \emph{the} Ramsey number of $G$.
The most notorious open problem here is to determine the Ramsey number of cliques.
The classical bounds on $R(K_n)$ by Erd\H os~\cite{Erds1947} and Erd\H os and Szekeres~\cite{erdos1935combinatorial} imply that
$\sqrt{2}^n \le R(K_n)\le 4^n$, for~$n\geq 3$,
so $R(K_n)$ is exponential in $n$, but despite tremendous efforts its exact behaviour remains unknown (see~\cite{CGMS2023} for the most recent improvements).

In general, if an $n$-vertex graph $G$ has $m$ edges and no isolated vertices, then $ 2^{\Omega(m/n)} \leq R(G) \leq 2^{O(\sqrt{m})}$, where the lower bound follows from a probabilistic construction and the upper bound was shown by Sudakov~\cite{Sudakov2011}.
Given additional structure on $G$, there are many cases where we can even obtain $R(G) = O(n)$.
This holds, for instance, for graphs with bounded maximum degree~\cite{CHVATAL1983},
bounded arrangeability~\cite{ChenSchelp1993},
or bounded degeneracy~\cite{lee2017ramsey}. We recommend~\cite{conlon2015recent} for a survey in the area. 

As we have seen, the Ramsey number of an $n$-vertex graph can vary between linear and exponential in $n$. 
A natural question is thus to ask which values (between $n$ and~$R(K_n)$) can be attained as the Ramsey number of some $n$-vertex graph.
The aim of this note is to study this question, and, in particular, to determine which functions~$f:\mathbb N\to\mathbb N$, with $n\le f(n)\le R(K_n)$ for all $n\in\mathbb N$, are the rate of growth of the Ramsey numbers of some sequence of $n$-vertex graphs.

It is natural here to restrict our analysis to connected graphs.
Note that after adding~$n-r$ isolated vertices to an $r$-vertex graph~$H$, we obtain an~$n$-vertex graph~$H'$ satisfying~$R(H')=\max\{n, R(H) \}$.
This means that we can obtain values for the Ramsey numbers of $n$-vertex graphs which in essence correspond to the Ramsey numbers of $r$-vertex graphs; restricting to connected graphs rules out such constructions.
Our first result is that every function, between the appropriate bounds, can be attained as the rate of growth of some sequence of graphs, up to a multiplicative factor.

\begin{theorem}\label{thm:main}
There exists a positive constant $C$ such that for every function $f:\mathbb{N}\to\mathbb{N}$, with $n\le f(n)\le R(K_n)$, there exists a sequence of connected graphs $(G_n)_{n\in\mathbb{N}}$ such that for all $n \in \mathbb{N}$, $|V(G_n)|=n$ and $f(n)\le R(G_n)\le Cf(n).$
\end{theorem}

In other words, Theorem~\ref{thm:main} states\footnote{An alternative way to phrase Theorem~\ref{thm:main} is that there exists an absolute constant $C>0$ such that for all $n\in\mathbb N$ and $n\le a\le R(K_n)$, there exists a connected graph $G$ on $n$ vertices such that $a\le R(G)\le Ca$. } that $R(G_n)=\Theta(f(n))$, where the implicit constants do not depend on the function $f$. We remark that by a result of Burr and Erd\H os~\cite{BurrErdos1976} on the Ramsey number of trees, it is known that every $n$-vertex connected graph $G$ satisfies $R(G) \geq \lceil\tfrac{4}{3}n\rceil - 1$; thus taking the function $f(n) = \alpha n$ for any $1 \leq \alpha < 4/3$ shows that the conclusion of Theorem~\ref{thm:main} cannot hold with $R(G_n) = (1 + o(1))f(n)$ instead.

Our second result concerns $k$-uniform hypergraphs. A \emph{$k$-graph} $H$ is a pair~$H\!=\!(V,E)$ where $V$ is the set of vertices of $H$ and every edge $e\in E$ is a $k$-element subset of $V$. For $n\in\mathbb N$, the $k$-uniform clique on $n$ vertices $K_n^{(k)}$ is the $k$-graph on $n$ vertices in which every $k$-element set of vertices is an edge.
Given a $k$-graph~$H$,  the Ramsey number $R(H)$ of~$H$ is the smallest number $n$ such that every red-blue colouring of the edges of $K_n^{(k)}$ yields a monochromatic copy of $H$.  

We prove that an analogue of Theorem~\ref{thm:main} fails for $k$-graphs if $k\geq 5$ (even without any kind of connectivity restrictions).

\begin{theorem}\label{thm:hyp}
Let $k\ge 5$.
There exists a non-decreasing function~$f:\mathbb N\to\mathbb N$ with $n \leq f(n)\leq R(K_n^{(k)})$, such that for all $c, C \!>\!0$ and any $n_0$, there is an $n\!>\! n_0$ such that
\[R(H)\le cf(n)\quad\text{ or }\quad R(H)\ge Cf(n)\]
for every $n$-vertex $k$-graph $H$.
\end{theorem}

Using our techniques we can also answer a question posed by DeBiasio~\cite{DeBiasioMO}, who asked for the existence of a sequence $(G_n)_{n\in\mathds N}$ of graphs where $R_2(G_n)$ is linear whilst~$R_3(G_n)$ is superlinear.
Similar differences in behaviour depending on the number of colours have been observed before in infinite graphs (see \cite[Section 10.1]{CorstenDebiasioMckenney2020}) and in $3$-graphs (see~\cite{hedgehogs}).

\begin{theorem} \label{theorem:colourblind}
    There exists a sequence $(G_n)_{n \in \mathbb N}$ of graphs such that $|V(G_n)| = n$,  $R_2(G_n) = O(n)$ and $R_3(G_n) = \Omega(n \log n)$.
\end{theorem}
Let us point out here that the graphs we construct for Theorem~\ref{theorem:colourblind} have isolated vertices. However, if we insist on sequences of connected graphs, we can get the following.

\begin{theorem} \label{theorem:colourblind2}
    There is a sequence $(G_n)_{n \in \mathbb N}$ of connected graphs such that~$|V(G_n)| = n$, $R_2(G_n) = O(n \log n)$ and $R_3(G_n) = \Omega(n \log^2 n)$.
\end{theorem}

\section{Proof of Theorem~\ref{thm:main}}\label{section:mainproof}

Conlon, Fox and Sudakov~\cite[Lemma 5.5]{ConlonFoxSudakov2020} used a result by Erd\H os and Szemer\'edi~\cite{ErdosSzemeredi1972} to prove that the Ramsey number of a dense graph cannot decrease by much under the deletion of one vertex.
Recently, Wigderson~\cite{Wigderson2022} investigated this phenomenon in sparser graphs.
A graph on $n$ vertices has density $d$ if it has exactly $d\binom{n}{2}$ edges.

\begin{lemma}[\cite{ConlonFoxSudakov2020}]\label{lemma:CFS} There exists a function $g: [0,1] \rightarrow \mathbb{R}$ such that for every graph $H$ of density at least $d$ and any graph $H'$ obtained by deleting a single vertex from $H$, we have
$R(H) \leq g(d) R(H')$.
\end{lemma}

In fact, in~\cite[Lemma 5.5]{ConlonFoxSudakov2020} it is claimed the statement is true with $g(d) = c\log(1/d)/d$, for an absolute constant~$c>0$.
However, such a function works for $d$ bounded away from $1$ only.
Their proof can be trivially changed to obtain Lemma~\ref{lemma:CFS}, which is all we need to show the following corollary.
In fact, \ref{it:cliques} was already noted in~\cite{ConlonFoxSudakov2020}.

\begin{lemma} \label{lemma:ramseyvertexdeletion}
    There exist constants $c_1, c_2 > 1$ such that for any $n \geq 1$,
    \begin{enumerate}
        \item $R(K_{n+1}) \leq c_1 R(K_n)$,\label{it:cliques}
        \item $R(K_{n+1, n+1}) \leq c_2 R(K_{n,n})$.
    \end{enumerate}
\end{lemma}

\newcommand{\constant}{c_2}
We also need the Ramsey number of a path $P_n$ with $n$ edges, determined by Gerencsér and Gy\'arfás~\cite{GerencserGyarfas1967}.

\begin{lemma} \label{lemma:ramseypath}
    For every $n \geq 1$, $R(P_n) = \lfloor (3n+1)/2 \rfloor$.
\end{lemma}

We shall also use a lower bound on the Ramsey number of complete bipartite graphs, which  follows from a standard probabilistic construction.

\begin{lemma}\label{lemma:lowerbound}For $t \geq 1$,
	$R(K_{t, t}) \geq 2^{t/2}$.
\end{lemma}

Theorem~\ref{thm:main} will be a direct consequence of the following two results.

\begin{lemma}\label{lemma:biclique-path}
    Suppose $1 \leq t \leq n/2$ and let $H_{n,t}$ be the graph formed by taking a copy of $K_{t,t}$ and attaching to one of its vertices a path on $n-2t$ new vertices.
    Then $$R(H_{n,t}) \leq 3R(K_{t,t})/2+3n\,.$$
\end{lemma}

\begin{proof}
    Let $H = H_{n,t}$ and
    let $N = \lfloor (3R(K_{t,t})+6n)/2\rfloor$. 
    Consider an arbitrary red-blue edge-colouring of $K_N$; we shall show that it contains a monochromatic copy of $H$. For a contradiction, assume it does not.
    By Lemma~\ref{lemma:ramseypath}, there exists a monochromatic path $P'$ in $K_N$ of length at least $R(K_{t,t})+2n-1$, and we assume without loss of generality that $P'$ is red.
    Let $P \subseteq P'$ be obtained after removing $n-2t$ vertices at one extreme of the path $P'$.
    Thus $P$ has at least $R(K_{t,t})+2n - (n-2t) \geq R(K_{t,t})+n$ vertices.
    Let $S = V(P)$.
    
    If $S$ contains a red copy of $K_{t,t}$, then together with $P'$ we can find in $K_N$ a red path of length at least $n-t$ joined to one of its vertices, a contradiction.
    Since $|S| \geq R(K_{t,t})+n$, we can find a monochromatic copy of $K_{t,t}$ in $S$, which must be blue.
    In fact, we can greedily find vertex-disjoint blue copies of $K_{t,t}$ until fewer than $R(K_{t,t})$ vertices remain uncovered.
    Let $K^1, \dotsc, K^s$ be the copies that we have found.
    Note that these copies together cover more than $|S| - R(K_{t,t}) \geq n$ vertices.
    
    For all $1 \leq i \leq s$, let $A_i, B_i$ be the two classes of $K^i$.
    Given $1 \leq i < s$, note that not all edges between $B_i$ and $A_{i+1}$ can be red, as that would yield a red monochromatic copy of $K_{t,t}$ in $S$.
    Therefore, there are blue edges $e_1, \dotsc, e_{s-1}$ where each $e_i$ has one endpoint $b_i\in B_i$ and other endpoint $a_{i+1}\in A_{i+1}$.
    Let $a_1 \in A_1$ be arbitrary.
    For all~$1 \leq i < s$, take a blue path $P_i \subseteq K^i$ which spans $V(K^i)$ and has endpoints $a_i$ and $b_i$.
    Thus, the concatenation $P_1 + e_1 + \dotsb + P_{s-1} + e_{s-1}$, together with $K^s$, forms a blue copy of $H$, a final contradiction.
\end{proof}

\begin{lemma}\label{lemma:clique-path}
    Suppose $2 \leq t \leq n$ and let $J_{n,t}$ be the graph formed by taking a copy of~$K_{t}$ and attaching to one of its vertices a path on $n-t$ new vertices.
    Then $$R(J_{n,t}) \leq 3(R(K_t)+(t+1)n)/2\,.$$
\end{lemma}

\begin{proof}
    Let $J = J_{n,t}$ and
    let $N = \lfloor{3(R(K_t)+(t+1)n)/2}\rfloor$.
    Consider an arbitrary red-blue edge-colouring of $K_N$; we shall show that it contains a monochromatic copy of $J$.
    For a contradiction, assume it does not.
    By Lemma~\ref{lemma:ramseypath}, there exists a monochromatic path $P'$ in $K_N$ of length at least $R(K_t)+(t+1)n - 1$, and we assume without loss of generality that $P'$ is red.
    Let $P \subseteq P'$ be obtained after removing $n-t$ vertices at one extreme of the path $P'$.
    Thus $P$ has at least $R(K_t)+(t+1)n - (n-t) \geq R(K_t)+nt$ vertices.
    Let $S = V(P)$.
    
    If $S$ contains a red copy of $K_{t}$, then together with $P'$ we can find in $K_N$ a red path of length at least $n-t$ joined to one of its vertices, a contradiction.
    Since $|S| \geq R(K_t)+nt$, we can find a monochromatic copy of $K_{t}$ in $S$, which must be blue.
    In fact, we can greedily find vertex-disjoint blue copies of $K_{t}$ until fewer than $R(K_{t})$ vertices remain uncovered.
    Let $Q^1, \dotsc, Q^s$ be the copies that were found.
    These copies, together, cover at least $|S| - R(K_{t}) \geq nt$ vertices of $S$, and thus we have $s \geq n$.
    
    Define a \emph{clique-path} to be a sequence of vertex-disjoint blue cliques $Q^{1}, \dotsc, Q^l$ such that for each $1 \leq i < l$ there is a blue edge $e_i$ between $Q^{i}$ and $Q^{i+1}$, and the edges~$e_i$ are vertex-disjoint for all $1 \leq i < l$.
        
    \begin{claim}
        There is a set of at most $t-1$ clique-paths that together cover all cliques $Q^1, \dotsc, Q^s$ exactly once.
    \end{claim}
    
    \begin{proof}
        Suppose otherwise and let $P_1, \dotsc, P_{t-1}$ be $t-1$ clique-paths which use pairwise-disjoint sets of cliques and together use the maximum possible number of cliques.
        For each $1 \leq i \leq t-1$, let $Q^i$ be an ``end-clique'' of each $P_i$.
        Let $Q^t$ be any clique not covered by $\{P_1, \dotsc, P_{t-1}\}$, which exists by assumption.
        In each $Q^1, \dotsc, Q^{t}$, we select a vertex $q_i$ which is not in any of the inter-clique edges of the clique-paths (here we use $t \geq 2$).
        Since $S$ contains no red~$K_t$, there must be a blue edge between some pair $q_i q_j$.
        But then we can obtain a new clique-path which contains $Q^i$, $Q^j$ and the edge $q_i q_j$.
        Thus we have found a new family of $t-1$ clique-paths covering one more clique, a contradiction.
    \end{proof}
    
    Therefore, there is a clique-path which uses at least $s/(t-1) \geq s/t$ cliques.
    In such a clique-path, we can easily find a blue clique $K_t$ together with a blue path which together use at least $t\cdot (s/t) = s \geq n$ vertices, as required.
\end{proof}

Now we are ready to prove the main result of this section.

\begin{proof}[Proof of Theorem~\ref{thm:main}]
    It is easily seen that Theorem~\ref{thm:main} follows immediately from the following statement about `gaps' in Ramsey numbers:
    there exists $C > 0$ such that for all $n, a \in \mathbb N$ and $n \leq a  \leq R(K_n)$, there exists a connected $n$-vertex graph $G$ with $a \leq R(G) \leq Ca$.
    From now on, we prove this latter statement.

    By the form of the statement, we can assume that $n$ is sufficiently large so that the inequalities that need it are true.
    Let $c_1$ and $c_2$ be the constants from Lemma~\ref{lemma:ramseyvertexdeletion} such that $R(K_t)\le c_1R(K_{t-1})$ and $R(K_{t,t})\le c_2R(K_{t-1,t-1})$ for all $t \geq 2$, and let $C$ be a sufficiently large constant. 

We will split the proof into two cases, depending on how large $a$ is.
In fact, the two ranges we consider are not disjoint, but they are enough to cover all possibilities between $n$ and $R(K_n)$.\medskip

\noindent\emph{Case 1: $n\le a\le 2^{n/8}$.}
Let $t$ be the minimal number such that $R(K_{t,t}) > a$.
We note that by the choice of $t$, we have $R(K_{t-1, t-1}) \leq a < R(K_{t,t})$.
By Lemma~\ref{lemma:lowerbound}, we have $2^{(t-1)/2} \leq a$ and thus $t \leq 2 \log_2 a+1$.
Since $a \leq 2^{n/8}$, we have that $t \leq n/4 + 1$.
Since we assume $n$ to be large, we can assume $2t \leq n$.
    Let $G_n=H_{n,t}$ be the graph as in Lemma~\ref{lemma:biclique-path}. Since $K_{t,t}\subseteq G_n$, we have $R(G_n)\ge R(K_{t,t})>a$. For the upper bound, using  Lemmata~\ref{lemma:ramseyvertexdeletion} and~\ref{lemma:biclique-path} we deduce that
\[R(G_n)\le \frac{3}{2}R(K_{t,t})+3n\le \frac{3}{2}c_2a+3n\le Ca.\]

\noindent\emph{Case 2: $n^2\leq a \leq R(K_n)$.}
Take $t$ minimal subject to~$R(K_t) \geq a$.
Clearly, such $t$ always exists and is at most $n$. Thus we have $R(K_{t-1}) < a \leq R(K_t)$.
Moreover, since $R(K_r) \geq 2^{r/2}$ holds for all $r$, we know that $t \leq \min \{ n,  2 \log_2 a + 1 \}$. Let $G_n=J_{n,t}$ be as in Lemma~\ref{lemma:clique-path} and note that, since $K_t\subseteq G_n$, we have  $R(G_n)\ge R(K_t)>a$. For the upper bound, using Lemmata~\ref{lemma:clique-path} and~\ref{lemma:ramseyvertexdeletion} we have
\[\begin{array}{lll}R(G_n)\le \frac{3}{2}(R(K_t)+(t+1)n)&\le& \frac{3}{2}(c_1a+(t+1)n)\\
&\le& Ca,\end{array}\]
where the last inequality follows from $(t+1)n \leq n(2 \log_2 a + 2) \le 3n\log_2a\le 6n^2\le 6a$.
\end{proof}

\section{Proof of Theorem~\ref{thm:hyp}}

For $k$-graphs, the so-called `stepping-up lemma' of Erd\H os, Hajnal, and Rado~\cite{erdoshajnal} allows us to deduce a tower-type lower bound for the Ramsey number~$R(K_n^{(k)})$ for every~$k\geq 3$, namely
\begin{align}\label{eq:Ramseyhyp}
    a n^2 \leq \log^{(k-2)}_2(R(K_n^{(k)})), 
\end{align}
where $a>0$ is a constant depending only on $k$ and $\log^{(i)}_2(\cdot)$ denotes the $i$th iterated base-$2$ logarithm (see \cite{ConlonFoxSudakov2013} for more explicit bounds of this form).

\begin{proof}[Proof of Theorem~\ref{thm:hyp}]
Let $k\!\geq\! 5$. 
We find a function~$g\colon\mathbb N\to \mathbb N$, with~$n\! \leq \!g(n)\!\leq\! R(K_n^{(k)})$ as follows. 
For every $n\in\mathbb N$, let $I_n=[\log_2 n, \log_2 R(K_n^{(k)})]$ be an interval in $\mathds{R}$.
Note that, since $k \geq 5$, inequality \eqref{eq:Ramseyhyp} implies that
\[\log_2 R(K^{(k)}_n) - \log_2 n \geq 2^{2^{a n}}-\log_2 n.\]
Since the number of $k$-graphs on $n$ vertices is at most $2^{n^k}$, by averaging we find a sub-interval~$I_n'\subseteq I_n$ which does not contain $\log_2 R(H)$ for any $n$-vertex $k$-graph $H$, and such that $I'_n$ has length at least
\begin{align*}
  \frac{2^{2^{a n}}-\log_2 n}{2^{n^k}+1}> 2n+1,
\end{align*}
where we used that $n$ is sufficiently large.
By passing to a sub-interval we may assume that $I'_n \cap \mathbb N$ has exactly $2n+1$ elements.
Let $m_n\in I'_n$ be the middle point of~$I_n' \cap \mathbb N$.
Then, for large $n$ and every $n$-vertex $k$-graph $H$ we have 
\begin{align}\label{eq:addemptyinter}
\log_2 R(H) \le m_n - n
\qquad \text{or} \qquad
\log_2 R(H) \ge  m_n + n.
\end{align}
Let $g:\mathbb N\to \mathbb N$ be defined by $g(n) = 2^{m_n}$.
Since $m_n\in I_n$ we have $n\leq g(n)\leq R(K_n^{(k)})$. Then, from~\eqref{eq:addemptyinter} we deduce that for every $n$ and every $n$-vertex $k$-graph $H$, 
\[ R(H)\leq 2^{-n} g(n)
     \qquad \text{or} \qquad 
     R(H)\geq 2^n g(n)\,.\]
In particular, for every two positive constants $c, C\!>\!0$ and for every sufficiently large~$n$, we have~$R(H) < c g(n)$ or~$R(H) > C g(n)$, as required.
     
Note that we could finish the proof here if we were only interested in gaps for hypergraph Ramsey numbers. However, if we insist on having a non-decreasing function, we may define $f:\mathbb N\to \mathbb N$  by setting $f(1)=g(1)$ and, for $n\ge 2$,
\[
f(n) = 
\begin{cases}
g(n) \quad &\text{ if }g(n) \geq f(n-1),\\ 
f(n-1) \quad &\text{ if }g(n) < f(n-1)\,.
\end{cases}
\]
It is straightforward to check that~$f$ is non-decreasing and satisfies the desired conditions. 
\end{proof}

Notice that the proof of Theorem~\ref{thm:hyp} relies on the fact that~$\log_2 R(K_n^{k}) = \omega(2^{n^k})$ for every $k\geq 5$. 
Erd\H os, Hajnal, and Rado \cite{erdoshajnal} conjectured that~\eqref{eq:Ramseyhyp} can be improved to $a n \leq \log^{(k-1)}_2(R(K_n^{(k)}))$
for every $k\geq 3$, in which case our proof of Theorem~\ref{thm:hyp} works for $4$-uniform hypergraphs as well.
The situation for $3$-uniform hypergraphs is not clear, even if this conjecture were true.

\section{Proofs of Theorems~\ref{theorem:colourblind} and~\ref{theorem:colourblind2}}

We shall use the following simple lemma.
We remark that similar statements have been obtained before, e.g., by Lefmann~\cite{Lefmann1987}.
For any graph $G$, let $\chi(G)$ be its chromatic number.

\begin{lemma} \label{lemma:colourblind-lower}
    For every graph $G$ and connected $H \subseteq G$, we have 
    $$R_3(G) \geq (\chi(H) - 1)(R_2(H) - 1) + 1\,.$$
\end{lemma}

\begin{proof}
    Let $N = (\chi(H) - 1)(R_2(H) - 1)$.
    We construct a red-blue-green colouring of $K_N$ as follows: partition $V(K_N)$ into $\chi(H) - 1$ sets $V_1, \dotsc, V_{\chi(H) - 1}$ of size $R_2(H) - 1$ each.
    Inside each $V_i$ use colours red and blue in such a way that the colouring does not contain a red-blue copy of $H$; colour every other edge green.
    
    This colouring does not contain a monochromatic copy of $G$.
    Indeed, a hypothetical such copy cannot be red or blue, as otherwise  there must exist a red or blue copy of $H$.
    Since $H$ is connected, such a copy of $H$ must lie inside one of the sets $V_i$, but we have chosen the red-blue edges so that this does not happen.
    Also, there are no green copies of $G_n$, since the graph formed by the green edges is $(\chi(H)-1)$-partite but $\chi(G) \geq \chi(H)$.
    We conclude that $R_3(G) > N$.
\end{proof}

\begin{proof}[Proof of Theorem~\ref{theorem:colourblind}]
Given $n \geq 2$, let $t$ be the least integer such that $n \leq R_2(K_t)$. Note that $t \geq 2$.
By choice, we have $R_2(K_{t-1}) < n$ and, by Lemma~\ref{lemma:ramseyvertexdeletion}, we have~$R_2(K_t) \leq c_1 R_2(K_{t-1}) < c_1 n$.
Let $G_n$ be the graph obtained from $K_t$ by adding $n-t$ isolated vertices, then $|V(G_n)| = n$ and $R_2(G_n) = \max\{ n, R_2(K_t) \} = R_2(K_t) < c_1 n = O(n)$.
On the other hand, since $n \leq R_2(K_t) \leq 4^t$, we know that $t \geq \frac{1}{2} \log_2 n$ and therefore by Lemma~\ref{lemma:colourblind-lower} we have $R_3(G_n) = \Omega(n \log n)$.
\end{proof}

\begin{proof}[Proof of Theorem~\ref{theorem:colourblind2}] 
Let~$t$ be the least integer so that $R_2(K_t)\ge n\log_2n$ and let~$G_n = J_{n,t}$.
Applying Lemma~\ref{lemma:clique-path}, we obtain a constant~$C>0$ such that
$$R(G_n) \leq \frac{3}{2}(R(K_t) + (t+1)n) \leq C n \log_2n\,,$$
where we use the bound~$R(K_t)\leq 4^t$ to deduce that $t=\Theta(\log_2(n))$. 
Since~$\chi(G_n) = \Omega(\log n)$, then, by Lemma~\ref{lemma:colourblind-lower} and the definition of $t$, we have  
\[R_3(G_n) > (\chi(G_n)-1)(R_2(G_n)-1) = \Omega(n \log^2 n),\]
as required.
\end{proof}

\section{Concluding remarks} \label{section:gaps}

\subsection{Gaps}
For $n\in\mathbb N$, let us consider the sets
\begin{align*}
\mathcal R_n &= \{R(G) \colon \vert V(G)\vert = n\}, \\
\mathcal R_n^{\circ} &= \{R(G) \colon G \text{ does not contain isolated vertices and } \vert V(G)\vert = n\}, \quad \text{and}\\    
\mathcal R_n^{\mathfrak{c}} &= \{R(G) \colon G \text{ is connected and } \vert V(G)\vert = n\}.
\end{align*}
It is clear that~$\mathcal R_n^{\mathfrak{c}} \subseteq \mathcal R_n^{\circ} \subseteq \mathcal R_n \subseteq [n, R(K_n)]$. 
Observe that $n\in \mathcal R_n$ since $R(\overline{K_n}) = n$, where $\overline{K_n}$ corresponds to an independent set on $n$ vertices.
Furthermore, consider a disjoint union of two stars~$\Sigma_{a,b} = K_{1,a}\cup K_{1,b}$. 
A result due to Grossman \cite{grossman} implies that $R(\Sigma_{a,a-i}) = 3a-2i$ for $i\in \{0,1,2\}$.
Thus, by adding~$n-(2a-i+2)$ extra isolated vertices to $\Sigma_{a,a-i}$ and letting  the value of~$a$ vary from $\lfloor n/3\rfloor$ to~$\lfloor(n-2)/2\rfloor$, we deduce that $[n, 3\lfloor\tfrac{n}{2}\rfloor-3]\subseteq \mathcal R_n$ for large $n\in\mathds N$.
Other families of sparse graphs can also be used to show other inclusions of this kind.

As mentioned in the introduction, $R(G)\geq \lceil\tfrac{4}{3}n\rceil-1$ holds for every connected graph~$G$ on~$n$ vertices, and this bound is tight. 
In particular, it implies that
\[\mathcal R^{\mathfrak{c}}_n\subseteq \Big[\big\lceil\tfrac{4}{3}n\big\rceil-1, R(K_n)\Big].\]
In a similar fashion, Burr and Erd\H{o}s~\cite{BurrErdos1976} showed that~$R(G)> n + \log n - O(\log\log n)$ holds for every $G \in \mathcal R_n^{\circ}$, which  is almost tight as shown by Csákány and Komlós~\cite{smallest}.
It would be interesting to get a better understanding of the structures of $\mathcal R_n$, $\mathcal R_n^{\circ}$, and~$\mathcal R_n^{\mathfrak{c}}$. 

Given a constant $c\!>\!1$,  say that a number $a\!\in\![n, R(K_n)]$ is a \emph{$c$-gap for $\mathcal R_n^{\mathfrak{c}}$} if~$[a,ca]\cap \mathcal R_n^{\mathfrak{c}}=\emptyset$.
As mentioned inside its proof, it is easy to see that Theorem~\ref{thm:main} is equivalent to the existence of a constant~$c\geq 1$ for which~$\mathcal R_n^{\mathfrak{c}}$ has no $c$-gaps for every sufficiently large~$n$.

In this direction, a proper (but non-empty) subset of the authors of this paper believes that the answer to the following question should be affirmative.
\begin{question}\label{quest}
Does there exist an $n_0\in \mathds N$ such that for every $n\geq n_0$
$$\mathcal R_n = [n,R(K_n)] \quad \text{and} \quad \mathcal R_n^{\mathfrak{c}} = \big[\lceil\tfrac{4}{3}n\big\rceil-1, R(K_n)\big]\,\text{?}$$
\end{question}
Observe that the first equality would imply that for every function~$f\colon \mathds N\to \mathds N$ with~$n\leq f(n)\leq R(K_n)$ there is a sequence of graphs $(G_n)_{n\in\mathds N}$ such that $f(n)=R(G_n)$ for sufficiently large $n\in\mathds N$, and an analogous statement would hold for connected graphs if the second identity was true. However, in order to show the first equality, one would need to show that $R(K_n-e)\in\{R(K_n)-1,R(K_n)\}$ for sufficiently large $n\in\mathds N$, which is likely to be a quite hard problem.
A simpler question, which we still think is interesting is as follows:

\begin{question}
    Given $c > 1$, it is true that there are no $c$-gaps $a$ for $\mathcal R_n^{\mathfrak{c}}$ with $a \geq 4n/3$, for all sufficiently large $n$?
\end{question}

\subsection{Chromatic number and connectivity}
We observe that the proof of the first case of Theorem \ref{thm:main} can be modified by replacing the r\^ole of $K_{t,t}$ with a complete $k$-partite graph $K_{t,\dots, t}$.
In this way, we may ensure that the graphs in the sequence~$(G_n)_{n\in \mathds N}$ have large chromatic number for sufficiently large $n$. 

\begin{theorem}\label{thm:chroamtic} For every $k\!\geq\! 2$, there are positive constants $c$, $C$, and $n_0 \in \mathds N$ such~ that for every  function $f : \mathbb{N}\to\mathbb{N}$, with $n\le f(n)\le R(K_n)$, there~ is~ a sequence of connected graphs $(G_n)_{n\in\mathbb{N}}$ with $|V(G_n)| = n$ such that~$cf(n)\le R(G_n)\le Cf(n)$ for all $n\in\mathbb{N}$. 
Moreover,~$\chi(G_n)\geq k$ for every $n\geq n_0$.
\end{theorem}

It would be interesting to ensure other properties for the graphs in this sequence.
In particular, we believe the graphs can also be taken to have large connectivity.

\begin{conjecture}
For every $k\geq 2$ and for every  function $f:\mathbb{N}\to\mathbb{N}$ with $n\le \!f(n)\!\le R(K_n)$, there is a sequence of graphs $(G_n)_{n\in\mathbb{N}}$ with $|V(G_n)|\!=\!n$ such that 
$R(G_n) = \Theta(f(n))$ and $G_n$ is $k$-connected for all $n$ sufficiently large.
\end{conjecture}

It is also natural to ask if \Cref{theorem:colourblind} holds if we require, in addition, that the graphs $G_n$ are connected.

\begin{question}
    Is there a sequence $\{G_n\}_{n \in \mathbb N}$ of connected graphs, with $|V(G_n)| = n$, such that $R_2(G_n) = O(n)$, but $R_3(G_n) = \omega(n)$?
\end{question}

\subsection{Hypergraphs}
We finish by asking what happens with $k$-graphs in the cases not covered by \Cref{thm:main} and \Cref{thm:hyp}, i.e. $k \in \{3,4\}$.
We phrase our question in terms of gaps.

\begin{question}
    Given $k \in \{3, 4\}$ and arbitrary $C > 0$, does there exist $n \in \mathbb N$, and $n \leq a \leq R(K^{(k)}_n)$ such that no $k$-graph $H$ satisfies $a \leq R(H) \leq Ca$?
\end{question}

\subsection*{Acknowledgements}
The authors thank Letícia Mattos and Louis DeBiasio for useful discussions and Pedro Araújo for unknowingly inspiring us.
We also thank the referees for useful suggestions throughout the paper. 
\sloppy\printbibliography

\end{document}